\documentclass[11pt]{amsart}
\usepackage{amssymb}
\usepackage{times}
\usepackage{url}
\newtheorem{theorem}{Theorem}[section]
\newtheorem{lemma}[theorem]{Lemma}
\newtheorem{proposition}[theorem]{Proposition}
\newtheorem{corollary}[theorem]{Corollary}
\theoremstyle{remark}
\newtheorem{remark}[theorem]{Remark}

\newcommand{\Z}{\mathbb{Z}}
\newcommand{\Q}{\mathbb{Q}}
\newcommand{\C}{\mathbb{C}}

\newcommand{\GL}{\mathrm{GL}}
\newcommand{\SL}{\mathrm{SL}}

\newcommand{\glnz}{\mathrm{GL}(n,\Z)}
\newcommand{\glnq}{\mathrm{GL}(n,\Q)}
\newcommand{\glnc}{\mathrm{GL}(n,\C)}
\newcommand{\glnr}{\mathrm{GL}(n,R)}
\newcommand{\glnp}{\mathrm{GL}(n,p)}

\newcommand{\hi}{\mathrm{h}}
\newcommand{\fis}{\leq_f \!}
\newcommand{\gpess}{\langle S \hspace{1pt} \rangle}
\newcommand{\abk}{\allowbreak}
\newcommand{\gl}{\mathfrak{\mathop{gl}}}
\newcommand{\mf}[1]{\mathfrak{#1}}
\newcommand{\g}{\mathfrak{g}}
\newcommand{\h}{\mathfrak{h}}

\begin{document}

\title{Integrality and arithmeticity of solvable linear groups}

\date{\today}

\begin{abstract}
We develop a practical algorithm to decide whether a finitely
generated subgroup of a solvable algebraic group $G$ is
arithmetic. This incorporates a procedure to compute a generating
set of an arithmetic subgroup of $G$. We also provide a simple new
algorithm for integrality testing of finitely generated
solvable-by-finite linear groups over the rational field. The
algorithms have been implemented in {\sc Magma}.
\end{abstract}

\author{W.~A. de Graaf, A.~S. Detinko and D.~L. Flannery}

\address{W.~A. de Graaf, Department of Mathematics,
University of Trento, Italy}

\address{A.~S. Detinko and D.~L. Flannery,
Department of Mathematics, National University of Ireland, Galway,
Ireland}

\maketitle

\section{Introduction}

For $K\leq \glnc$ and a subring $R$ of $\C$, denote $K\cap \glnr$
by $K_R$. Let $G$ be an algebraic group defined over $\Q$. This
paper is concerned with the question
\[
\text{Given a finitely generated subgroup $H$ of $G_\Q$, is $H$ an
arithmetic subgroup of $G$?} \tag{$\ast$}
\]
Recall that $H\leq G_\Q$ is an arithmetic subgroup of $G$ if it is
commensurable with $G_\Z$, i.e., $H \cap \abk G_\Z= H_\Z$ has
finite index in both $H$ and $G_\Z$. (More generally, a matrix
group is said to be arithmetic if it is an arithmetic subgroup of
some algebraic $\Q$-group \cite[p.~119]{Segal}.) The significance
of $(\ast)$ is evidenced by, e.g., \cite{Sarnak}. Decidability of
($\ast$) has not been settled. However, it seems to be
undecidable, even for $G=\abk \SL(n,\C)$ \cite{Miller}. We prove
that ($\ast$) is decidable \emph{when $G$ is solvable}, by giving
a practical algorithm to answer this question.

As further motivation for arithmeticity testing we observe that a
positive answer enables the use of arithmetic group theory to
investigate the group at hand. For example, the automorphism group
of a finitely generated nilpotent group is isomorphic to an
arithmetic group \cite[Corollary 9, p.~122]{Segal}. There are
polycyclic groups that are not isomorphic to any arithmetic group
\cite[Proposition 3, p.~259]{Segal}. Although we can test whether
a polycyclic subgroup of $G_{\Q}$ is arithmetic, testing
arithmeticity of polycyclic groups in this broader context is
still open.

Our approach draws on recent progress \cite{deGraafI, deGraafII}
in the construction of generating sets of $G_\Z$ when $G$ is
unipotent or a torus. We combine those to construct a finite index
subgroup of $G_\Z$ for a solvable algebraic group $G$---a result
of interest in its own right.

We adapt methods for computing with SF (solvable-by-finite) linear
groups \cite{Tits, Recog, SF}. To decide whether $H$ is
commensurable with $G_\Z$, we use an algorithm from \cite{SF} to
compute the rank of an SF linear group. We also design a simple
new algorithm to test whether a finitely generated SF subgroup of
$\glnq$ is conjugate to a subgroup of $\glnz$. Since integrality
is an important linear group property \cite[Theorem~3.5,
pp.~54--55]{DixonBook}, this is another useful result of the
paper. Integrality testing appears as a component of earlier
algorithms to test finiteness \cite{BBR} and polycyclicity
\cite{AssmanEickII}.

The paper is organized as follows. Section~\ref{Willem} gives a
procedure to construct a generating set of a finite index subgroup
of $G_\Z$, where $G$ is solvable algebraic. In
Section~\ref{IntegInter} we consider testing finiteness of
$|K:K_\Z|$ for $K\leq \glnq$. If this index is finite, we explain
how to find $K_\Z$ and $g\in \abk \GL(n,\Q)$ such that $K^g \leq
\GL(n,\Z)$. Section~\ref{IntegSFQ} then describes a new algorithm
to test integrality of SF subgroups of $\glnq$. Our main algorithm
is presented in Section~\ref{ArithTest}. A discussion of
experimental results derived from a {\sc Magma} \cite{Magma}
implementation of the algorithms concludes the paper.

\section{Computing an arithmetic subgroup of a solvable algebraic group}
\label{Willem}

Let $G$ be a solvable algebraic $\Q$-group. In this section we
combine results from \cite{deGraafI} and \cite{deGraafII} to
construct a generating set of a finite index subgroup of $G_\Z$.

The first result is \cite[Proposition 7.2 (3)]{vinbergv}.
\begin{proposition}\label{thm:sd}
Suppose that $G = A \ltimes N$, where both $A$ and $N$ are
algebraic subgroups of $G$, defined over $\Q$. Let
$\Gamma_A,\Gamma_N$ be arithmetic subgroups of $A$ and $N$
respectively. Then $\Gamma_A\ltimes \Gamma_N$ is an arithmetic
subgroup of $G$.
\end{proposition}

From \cite[Ch.~5, \S 3, No.~5, Propositions 20 and 21]{cheviii} we
get the following.
\begin{proposition}\label{thm:solv}
Let $G$ be connected and solvable, with Lie algebra $\g \subset
\gl(n,\C)$. Let $\mf{n}$ be the ideal of $\g$ consisting of all
nilpotent elements of $\g$. Then there is a subalgebra $\mf{d}$ of
$\g$, consisting of commuting semisimple elements, such that $\g =
\mf{d}\oplus \mf{n}$. Moreover, both $\mf{n}$, $\mf{d}$ are
algebraic, and if we let $N$, $D$ be the corresponding connected
subgroups of $G$, then $G = D\ltimes N$.
\end{proposition}

Let $G$, $\g$ be as in Proposition~\ref{thm:solv}. We suppose that
$\g$ is given by a basis consisting of matrices having entries in
$\Q$. In \cite{gra13}, an algorithm that computes bases of
subalgebras $\mf{d}$, $\mf{n}$ with the properties of
Proposition~\ref{thm:solv} is given. Here we do not go into the
details but only briefly recall the basic steps.
\begin{enumerate}
\item Compute a Cartan subalgebra $\h$ of $\g$. \item Let
$a_1,\ldots,a_r$ be a basis of $\h$, and let $a_i = s_i+n_i$ be
the Jordan decomposition of $a_i$. \item Let $\mf{d}$ be the space
spanned by $s_1,\ldots,s_r$ and let $\mf{n}$ be the space spanned
by $n_1,\ldots,n_r$ along with the Fitting $1$-component
$\g_1(\h)$. (The latter is the space $\cap_{i\geq 1} [ \h^i, \g]$,
where $[\h^i,\g] = \abk [\h,[\h,\ldots,[\h,\g]\ldots ]]$ ($i$
factors $\h$).)
\end{enumerate}

Starting from the given basis of $\g$, we compute generators of
an arithmetic group in $G_\Z$ by the following.

\vspace*{15pt}

${\tt GeneratingArithmetic}\hspace{.25pt} (G)$

\vspace*{1mm}

Input: a basis of the Lie algebra $\g$ of the solvable algebraic
group $G$.

Output: 
 a generating set for a finite index subgroup of $G_\Z$.

\begin{enumerate}
\item Compute bases of $\mf{d}$ and $\mf{n}$ as above. Denote by
$D$, $N$ the connected subgroups of $G$ with respective Lie
algebras $\mf{d}$, $\mf{n}$. \item Compute generators of $N_\Z$
using the algorithm from \cite{deGraafI}, and generators of $D_\Z$
using the algorithm from \cite{deGraafII}. (These algorithms take
as input bases of $\mf{d}$ and $\mf{n}$ respectively.) \item
Return the union of the two generating sets obtained in the
previous step.
\end{enumerate}

\vspace*{15pt}

We note that ${\tt GeneratingArithmetic}$ is correct. Indeed, let
$H$ be the group generated by its output. Then $H \leq G_\Z$. On
the other hand, $H$ is an arithmetic subgroup of $G$ by
Proposition~\ref{thm:sd}. So $H$ has finite index in $G_\Z$.

${\tt GeneratingArithmetic}(G)$ forms a vital part of our main
algorithm. Notice that $G$ need not be connected. Indeed, let
$G^\circ$ be the connected component of the identity of $G$. Since
$|G_\Z:G_\Z^\circ|$ is finite, ${\tt
GeneratingArithmetic}(G^\circ)$ returns a generating set of an
arithmetic subgroup of $G$.

\section{Integrality and $\mathrm{GL}(n,\mathbb{Z})$-intercepts}
\label{IntegInter}

This section and the next depend on some ideas from \cite{BBR}.

An element or subgroup of $\glnq$ that has a conjugate in $\glnz$
is said to be \emph{integral}. Let $H\leq \glnq$.

\begin{lemma}\cite[Proposition 2.3]{BBR}
\label{BBRspLemma} $H$ is integral if and only if there exists a
positive integer $d$ such that $dh \in \abk \mathrm{Mat}(n,\Z)$
for all $h\in H$.
\end{lemma}
\begin{proof}
Suppose that $H^g \leq \abk \GL(n,\Z)$ for some $g\in \glnq$. Then
$m^2H \subseteq \mathrm{Mat}(n,\Z)$ where $m$ is any common
multiple of the denominators of the entries in $g$ and $g^{-1}$.

Suppose that $dH \subseteq \mathrm{Mat}(n,\Z)$. Let $g$ be any
matrix whose columns comprise a $\Z$-\abk basis of the lattice
generated by $dH\Z^n$. Then $g \in \GL(n,\Q)$ and $H^g \leq \abk
\GL(n,\Z)$.
\end{proof}

Call $d= d(H)$ as in Lemma~\ref{BBRspLemma} a \emph{common
denominator} for $H$. As the above proof demonstrates, knowing
$d(H)$ is equivalent to knowing $g\in \glnq$ such that $H^g\leq
\glnz$. A method to construct $d$ may be extracted from
\cite[Section 2]{BBR}; we give an algorithm that is a modification
of this for our purposes in Section~\ref{IntegSFQ}. If $H$ is
finitely generated then we calculate $g$ from $d$ by means of the
following (cf.~\cite[Section 3]{BBR}).

\vspace*{15pt}

${\tt BasisLattice}\hspace{.25pt} (S, d)$

\vspace*{1mm}

Input: a finite set $S \subseteq \GL(n,\Q)$ and $d = d(H)$,
$H=\gpess$.

Output: a basis for the lattice generated by $dH\Z^n$ in $\Z^n$.

\vspace*{1mm}

\begin{enumerate}
\item $\mathcal L:= d\Z^n$.

\item While $\exists \ h \in S\cup S^{-1}$ such that $h\mathcal{L}
\not \subseteq \mathcal{L}$ do

\item[] \hspace*{20pt} $\mathcal{L}:=$ the lattice generated by
$\mathcal{L} \cup h\mathcal{L}$.

\item Return a basis of $\mathcal L$.

\end{enumerate}

\vspace*{15pt}

We write $L\fis H$ to indicate that the subgroup $L$ has finite
index in $H$.
\begin{lemma}\label{FIIntegralIff}
Let $H_1\fis H$. Then $H$ is integral if and only if $H_1$ is
integral.
\end{lemma}
\begin{proof}
Suppose that $H_1$ is integral. By Lemma~\ref{BBRspLemma},
$d_1H_1\subseteq \mathrm{Mat}(n,\Z)$ for some $d_1\in \Z$. Choose
a transversal $\{ h_1 , \ldots, \abk h_r\}$ for the cosets of
$H_1$ in $H$, and let $d_2$ be a positive integer such that
$d_2h_i \in \mathrm{Mat}(n,\Z)$ for all $i$. Since
$d_1d_2H\subseteq \mathrm{Mat}(n,\Z)$, $H$ is integral by
Lemma~\ref{BBRspLemma} again.
\end{proof}

\begin{lemma}\label{IfConjugateThenFI}
Suppose that $d$ is a common denominator for $H$, and let
$\mathcal{L}$ be the lattice generated by $dH\Z^n$.
\begin{itemize}
\item[{\rm (i)}] $\mathcal{L} \subseteq \Z^n \subseteq \frac{1}{d}
\mathcal{L}$.

\item[{\rm (ii)}] $H$ acts by left multiplication on the (finite)
set of lattices lying between $\mathcal{L}$ and
$\frac{1}{d}\mathcal{L}$.

\item[{\rm (iii)}] $H_\Z$ is the stabilizer of $\, \Z^n$ under the
action in \rm{(ii)}.
\end{itemize}
\end{lemma}
\begin{proof}
(i) \phantom{} Clear: $dH \subseteq \mathrm{Mat}(n,\Z)$ and $d\Z^n
\subseteq dH\Z^n \subseteq \mathcal L$.

(ii) \phantom{} We have $H\mathcal L = \mathcal L$. Thus $H$ acts
on the set of lattices $\mathcal{L}'$ such that $\mathcal{L}
\subseteq \mathcal{L}' \subseteq \frac{1}{d}\mathcal{L}$.

(iii) \phantom{} Let $h\in H$. If $h\in H_\Z$ then $\Z^n =
h(h^{-1}\Z^n) \subseteq h\Z^n$, so $h\Z^n=\Z^n$. Conversely, if $h
\Z^n = \abk \Z^n$ then every entry in $h$ is an integer.
\end{proof}

\begin{proposition}\label{FinIndexZpointsCriterion}
$H$ is integral if and only if $H_\Z\fis H$.
\end{proposition}
\begin{proof}
This is a consequence of Lemmas~\ref{FIIntegralIff} and
\ref{IfConjugateThenFI}.
\end{proof}

The following procedure constructs $H_\Z$ if $H$ is finitely
generated and $d(H)$ is known.

\vspace*{15pt}

${\tt IntegralIntercept}\hspace{.25pt} (S, d)$

\vspace*{1mm}

Input: a finite set $S \subseteq \GL(n,\Q)$ and $d = d(H)$,
$H=\gpess$.

Output: a generating set of $H_\Z$.

\vspace*{1mm}

\begin{enumerate}

\item $\mathcal{L}:=$ the lattice generated by ${\tt
BasisLattice}\hspace{.25pt} (S, d)$.

\item $\Lambda :=$ the set of all lattices $\mathcal L'$ such that
$\mathcal{L} \subseteq \mathcal{L}'\subseteq
\frac{1}{d}\mathcal{L}$.

\item \label{StabCalc} Return a generating set for the stabilizer
of $\Z^n$ under the action of $H$ on $\Lambda$.
\end{enumerate}

\vspace*{10pt}

\begin{remark}\label{HowToComputeZPoints}
Step \eqref{StabCalc} may be carried out using standard algorithms
for finite permutation groups to obtain a transversal for $H_\Z$
in $H$, then writing down Schreier generators. More practical
approaches are possible in special situations; say for polycyclic
$H$ (e.g., arithmetic $H\leq G_\Q$ and solvable $G$). In that case
an algorithm can be based on an orbit-stabilizer approach. This is
similar to the algorithm described in \cite{sogos}. Such an
algorithm enumerates the orbit of $\Z^n$ under action by $H$, in
the process obtaining generators of the stabilizer. The efficiency
of this approach depends heavily on orbit size.
\end{remark}

\section{Integrality of solvable-by-finite subgroups
of $\mathrm{GL}(n,\mathbb{Q})$}\label{IntegSFQ}

We next describe how to test integrality of a finitely generated
SF subgroup of $\glnq$, and compute a common denominator if the
group is integral.

Let $H = \langle h_1, \ldots , h_r\hspace*{-1pt} \rangle \leq
\glnq$. We have $H\leq \glnr$ where $R= \frac{1}{b}\Z = \{ a/b^i
\, | \, a \in \Z, \, \abk i\geq 0\}$ for some positive integer $b$
determined by the entries in the $h_i$ and $h_i^{-1}$. For any
prime $p\in \abk \Z$ not dividing $b$, reduction modulo $p$ of
matrix entries defines a congruence homomorphism $\varphi_p: \abk
\glnr\rightarrow \abk \GL(n,p)$. If $H$ is SF and $p\neq \abk 2$
then $H_p=\ker \varphi_p\cap H$ is torsion-free and
unipotent-by-abelian (see \cite[Lemma 9]{DixonOrbitStabilizer} or
\cite[Section 2.2.1]{Tits}). Assume that $p$ has been so chosen
whenever $H$ is SF.

Denote the unipotent radical of $K\leq \glnq$ by $U(K)$. Replace
$K$ if necessary by a conjugate in block triangular form with
completely reducible diagonal blocks. If $\pi$ denotes projection
of $K$ onto the block diagonal then $U(K)= \ker \pi$.
\begin{lemma}\label{VariousEqu}
The following are equivalent. \begin{itemize} \item[{\rm (i)}] $H$
is integral.  \item[{\rm (ii)}] $\pi(H)$ is integral. \item[{\rm
(iii)}] $\pi(H_p)$ is integral.\end{itemize}
\end{lemma}
\begin{proof}
By \cite[Theorem~2.4]{BBR}, a finitely generated subgroup $K$ of
$\glnq$ is integral if and only if $\{ \mathrm{tr}(x) \mid x \in
K\} \subseteq \Z$. Thus (i) $\Leftrightarrow$ (ii). Since
$|\pi(H):\pi(H_p)| = |H:H_p \hspace*{1.5pt} U(H)|<\infty$,
Lemma~\ref{FIIntegralIff} gives (ii) $\Leftrightarrow$ (iii).
\end{proof}

Denote the characteristic polynomial of $h\in \glnc$ by
$\chi_h(X)$.
\begin{proposition}[\mbox{Cf.} Lemma 8 and Theorem 9 of
\cite{AssmanEickII}]
\label{CharPolyCrit} Suppose that each element of $H$ is equal to
$h_1^{m_1} \cdots h_r^{m_r}$ for some $m_i\in\Z$. Then $H$ is
integral if and only if each $h_i$ is integral, i.e.,
$\chi_{h_i}(X)\in \abk \Z[X]$ and $\mathrm{det}(h_i)=\pm 1$ for
$1\leq i\leq r$.
\end{proposition}
\begin{proof}
If $\chi_{h_i}(X)\in\Z[X]$ has constant term $\pm 1$ for all $i$
then $\langle h_i \rangle \subseteq \big\{
\textstyle{\sum_{j=0}^{n-1}}a_jh_i^j \mid a_j \in \Z\big\}$. So
there exist positive integers $d_i$, $1\leq i\leq r$, such that
$d_i \langle h_i \rangle \subseteq \mathrm{Mat}(n,\Z)$. Hence $d=
d_1 \cdots d_r$ is a common denominator for $H$.
\end{proof}

If $H$ is polycyclic with polycyclic sequence $(h_1, \ldots ,
h_r)$, then $H$ satisfies the hypothesis of
Proposition~\ref{CharPolyCrit}. Any generating set of $H$ is
similarly suitable when $H$ is abelian.
\begin{lemma}\label{NormGenIntegralSuffices}
Suppose that $H\leq K\leq \glnq$ and the normal closure $H^K$ is
finitely generated abelian. Then $H^K$ is integral if and only if
each $h_i$ is integral.
\end{lemma}
\begin{proof}
If the $h_i$ are integral then Proposition~\ref{CharPolyCrit}
guarantees that $H$ is integral. Since $H^K$ is finitely
generated, there are $x_1,\ldots, x_t \in H$ and $y_1, \ldots ,
y_t \in K$ such that $H^K = \langle x_{i}^{y_i} : 1\leq i \leq
t\rangle$. By Proposition~\ref{CharPolyCrit} again, $H^K$ is
integral.
\end{proof}
\begin{lemma}\label{VeryStraightforward}
Suppose that $H$ is SF and $Y$ is a normal generating set for
$H_p$ ($p\neq 2$), i.e., $H_p=\abk \langle Y\rangle^H$. Then $H$
is integral if and only if each element of $\hspace*{.5pt} Y$ is
integral.
\end{lemma}
\begin{proof}
If each element of $Y$ is integral then the same holds for
$\pi(Y)$. Hence the finitely generated abelian group $\pi(H_p) =
\langle \pi(Y)\rangle^{\pi(H)}$ is integral by
Lemma~\ref{NormGenIntegralSuffices}. The claim now follows from
Lemma~\ref{VariousEqu}.
\end{proof}
We compute $Y$ from a presentation $\mathcal P$ of
$\varphi_p(H)\leq \glnp$ on $\varphi_p(h_1), \ldots ,
\varphi_p(h_r)$ using the function ${\tt NormalGenerators}$ as in
\cite[Section 3.2]{Tits}: this evaluates the relators of $\mathcal
P$, replacing $\varphi_p(h_i)$ everywhere by $h_i$, $1\leq i\leq
r$. Proposition~\ref{CharPolyCrit} and
Lemma~\ref{VeryStraightforward} consequently provide our
straightforward procedure to test integrality of $H$.

\vspace*{15pt}

${\tt IsIntegralSF} (S)$

\vspace*{1mm}

Input: a finite subset $S$ of $\GL(n,\Q)$ such that $H=\gpess$ is
SF.

Output: ${\tt true}$ if $H$ is integral; ${\tt false}$ otherwise.

\vspace*{1mm}

\begin{enumerate}

\item $Y:= {\tt NormalGenerators}\hspace{.5pt} (S)$.

\item If every element of $Y$ is integral then return ${\tt
true}$;

else return ${\tt false}$.
\end{enumerate}

\vspace*{10pt}

\begin{remark}
When $H$ is finite we have $H_p=1$, so ${\tt IsIntegralSF}$ always
returns ${\tt true}$ for such input.
\end{remark}

A major class of SF groups is PF (polycyclic-by-finite) groups.
According to \cite[Corollary~10]{AssmanEickII}, one may test
integrality of a PF subgroup $H$ of $\glnq$ after computing a
polycyclic sequence and transversal for a finite index polycyclic
normal subgroup of $H$. By contrast, ${\tt IsIntegralSF}$ does not
require $H$ to be PF and just tests integrality of several
matrices in $H_p$.

To conjugate an integral SF group $H$ into $\GL(n,\Z)$, or to
compute generators of its finite index subgroup $H_\Z$ via ${\tt
IntegralIntercept}$, we must know $d(H)$. The method below to find
this common denominator is a simplification of \cite[p.~120]{BBR}
for SF input.

First determine a  block upper triangular conjugate of $H$ with
completely reducible diagonal blocks; this may be done as in
\cite[Section~4.2]{SF}. Let $\{ a_1, \ldots , a_m \}\subseteq
\pi(H)$ be a basis of the enveloping algebra $\langle
\pi(H)\rangle_{\Q}$. If $c$ is a common multiple of the
denominators of entries in the $a_i$ then $d_1 :=  c\,
\mathrm{det}\big([\mathrm{tr}(a_ia_j)]_{1\leq i,j\leq n}\big) =
\abk d(\pi(H))$. Define $u_i = \abk h_i - \pi(h_i)$ and $v_i =
\abk h_i^{-1} - \pi(h_i^{-1})$. Let $d_2 = \abk e^{n-1}$ where $e$
is a common multiple of denominators of entries in the $u_i$ and
$v_i$. Each element of $H$ is a sum of terms $x= \abk \pi(g_1)
w_1\cdots \pi(g_k)w_k$ where $g_j\in H$ and $w_j=1$ or some $u_i$
or $v_i$. Since $\pi(h)u_j$ and $\pi(h)v_j$ are nilpotent, if
there are at least $n$ occurrences of $u_j$s and $v_j$s in $x$
then $x=\abk 0$. Thus $dH\subseteq \abk \mathrm{Mat}(n,\Z)$ where
$d= \abk d_1^{n} d_2$. Note that for completely reducible $H$
(e.g., $H$ is finite), $d(H) =d(\pi(H)) = d_1$.

\section{Arithmeticity testing in solvable algebraic groups}
\label{ArithTest}

In this section we apply results of \cite{SF} and the previous
sections to test whether a finitely generated subgroup $H\leq
G_\Q$ of a solvable algebraic group $G$ is arithmetic.

Denote the Hirsch number of a group $K$ with finite torsion-free
rank by $\hi(K)$. Finitely generated SF subgroups of $\GL(n, \Q)$
have finite torsion-free rank \cite[Proposition~2.6]{SF}.
\begin{lemma}\label{Step31}
Let  $L$ be a finitely generated SF subgroup of $\GL(n,\Q)$, and
let $K\leq L$. Then $K\fis L$ if and only if $\, \hi(K) = \hi(L)$.
\end{lemma}
\begin{proof}
This follows from Proposition~2.3 and Corollary~3.3 of \cite{SF}.
\end{proof}

\begin{corollary}\label{Step32}
$H$ is an arithmetic subgroup of $G$ if and only if $H$ is
integral and $\hi(H) = \hi(G_\Z)$. In particular, $H\leq G_\Z$ is
arithmetic if and only if $\, \hi(H) =\abk \hi(G_\Z)$.
\end{corollary}
\begin{proof}
%
This is immediate from Proposition~\ref{FinIndexZpointsCriterion}
and Lemma~\ref{Step31}, using that $\hi(H) = \hi(H_\Z)$ when
$H_\Z\fis \abk H$.
\end{proof}

\begin{remark}
If $G$ is a unipotent $\Q$-group then we have a more general
statement \cite[Lemma 3.7]{SF}: $H\leq G_\Q$ is arithmetic in $G$
if and only if $\hi(H ) = \hi(G_\Q)$, i.e., $\hi(H)$ equals the
dimension of $G$. For non-unipotent solvable (even abelian) $G$
this is not true; $G_\Q$ need not even have finite rank.
\end{remark}

We now state our arithmeticity testing algorithm.
This uses the procedure ${\tt HirschNumber}$ from \cite[Section
4.4]{SF}, which returns $\hi(K)$ for a finitely generated SF
subgroup $K$ of $\glnq$.

\vspace*{15pt}

${\tt IsArithmeticSolvable}\hspace{.25pt} (S, G)$

\vspace*{1mm}

Input: a finite subset $S$ of $G_\Q$,  $G$ a solvable algebraic
group.

Output: ${\tt true}$ if $H=\gpess$ is arithmetic; ${\tt false}$
otherwise.

\vspace*{1mm}

\begin{enumerate}

\item If ${\tt IsIntegralSF}\hspace{.25pt} (S) = {\tt false}$ then
return ${\tt false}$.

\item $T := {\tt GeneratingArithmetic}\hspace{.25pt} (G)$.

\item If ${\tt HirschNumber}\hspace{.25pt} (S)\neq {\tt
HirschNumber}\hspace{.25pt} (T)$ then return ${\tt false}$;

else return ${\tt true}$.
\end{enumerate}

\vspace*{10pt}

\begin{remark}
Steps (1) and (3) are justified by Corollary~\ref{Step32}. If $G$
is unipotent then $H$ is integral \cite[Lemma 2, p.~111]{Segal},
and step (1) can be omitted.
\end{remark}

\section{Practical performance}\label{ImplDisc}

We have implemented our algorithms in {\sc Magma}. The
implementation relies on the package `Infinite' \cite{Infinite}
and procedures available at \cite{deGraafPrograms}. Although the
main goal has been to establish that arithmeticity is decidable,
in this section we show that our algorithms can be applied in
practice to nontrivial examples.

One of the main bottlenecks of ${\tt GeneratingArithmetic}$
lies in the computation of the torus part. Let $\g=\mf{d}\oplus
\mf{n}$ and $D$ be as in Proposition \ref{thm:solv}. The first
step of the algorithm given in \cite{deGraafII} for computing
generators of $D_\Z$ constructs the associative algebra $A$ with
unity generated by $\mf{d}$. Subsequently $A$ is written as a
direct sum of number fields $\Q(\alpha)$. For each such field,
generators of the unit group of $\Z[\alpha]$ are computed. But the
algorithm for this task (as implemented in {\sc Magma}) becomes
extremely difficult to apply in practice if the degree of
$\Q(\alpha)$ is too large, say $30$ or more. On the other hand, if
all fields that occur are equal to $\Q$ then the computation of
generators of $D_\Z$ is trivial. For these reasons we constructed
test examples where the field extensions have degree $\leq 2$ (and
some extensions of degree $2$ do occur). This construction works
as follows. First we define a Lie algebra $\g(n) \subset
\gl(2n,\C)$, $n\geq 2$. For this we divide the matrices of
$\gl(2n,\C)$ into $2\times 2$ blocks. Our Lie algebra $\g(n)$ is
the direct sum $\g(n) = \mathfrak{t}\oplus \mathfrak{n}$ of two
subalgebras. The subalgebra $\mathfrak{t}$ has dimension $n$, and
the $i$th basis element has on its $i$th diagonal block the matrix
\[
\begin{pmatrix} 0 & 1 \\ 2i-1 & 0 \end{pmatrix}
\]
and zeros elsewhere.

Let $e_{i,j}$ be the elementary matrix with $1$ in position
$(i,j)$ and zeros elsewhere. Then $\mathfrak{n}$ is spanned by the
$e_{i,j}$ where $(i,j)$ appears in a block above the diagonal. So
$\dim \mathfrak{n} = 4 \binom{n}{2}$.

The Lie algebra $\g(n)$ is algebraic, and we let $G(n)< \GL(2n,\C)$
be the connected algebraic group with Lie algebra $\g(n)$.

Now let $m$ be a $2n\times 2n$ matrix produced by the following
randomized construction. The entries of $m$ are randomly and
uniformly chosen from $[0,0,1]$ (so we make it twice as likely
that a $0$ is chosen). We continue to produce matrices like this
until the determinant is not $0$ or $\pm 1$.

Set $\widetilde{G}(n) = mG(n)m^{-1}$. This is an algebraic group
with Lie algebra $m\g(n)m^{-1}$. The Lie algebra has basis $B_n$
consisting of all $mum^{-1}$, for $u$ in the above constructed
basis of $\g(n)$. Let $g_1,\ldots,g_r$ be generators of an
arithmetic subgroup of $G(n)_\Z$. For $1\leq i\leq r$ let $k_i$ be
chosen randomly and uniformly from $\{1,2\}$. Then $S = \{
(mg_im^{-1})^{k_i} \mid 1\leq i\leq r\}$ generates a subgroup of
$\widetilde{G}(n)_\Q$ (and note that it always is arithmetic). The
input on which we tested our implementation of ${\tt
IsArithmeticSolvable}$ is the set $S$, together with the algebraic
group $\widetilde{G}(n)$ given by its Lie algebra, in turn given
by its basis $B_n$.

In Table~\ref{tab:1} we report on the running times of our
algorithm with the input as above for $n=2$, $3$, $4$. All
experiments were performed on a 3.16~GHz machine running Magma
V2.19-9.
%
%
The first three columns of Table \ref{tab:1} list $n$, $\dim
\g(n)$, and the Hirsch number of $\widetilde{G}(n)_\Z$. The other
columns list running times of ${\tt IsIntegralSF}(S)$, $T:={\tt
GeneratingArithmetic}(\widetilde{G}(n))$, ${\tt HirschNumber}(S)$
and ${\tt HirschNumber}(T)$.

From Table~\ref{tab:1} we see that our algorithm is efficient
enough to handle quite nontrivial examples. Moreover, the bulk of
the running time goes into computing Hirsch numbers. Of course,
this depends on our particular test example. If we took examples
where it is difficult to compute generators of $D_\Z$, then a much
larger proportion of the time would go into computing an
arithmetic subgroup. However, the current implementation is rather
sensitive to randomness of the input. Occasionally it happens that
generators in $T$
%
have coefficients with many digits (up to $10$, for example). This
then causes problems when computing the Hirsch number. So we have
averaged times over $50$ runs in order to dampen the effects of
randomness of the input.


\bigskip

{\renewcommand{\arraystretch}{1.2}
\begin{table}[h]\caption{Running times (in seconds)
of the steps in  {$\tt IsArithmeticSolvable$}}\label{tab:1}
\begin{center}
\begin{tabular}{|c|c|c|c|c|c|c|}\hline
$n$  & $\mathrm{dim}\, \mathfrak{g}(n)$ &
$\hi(\widetilde{G}(n)_\Z)$ & ${\tt IsInt}(S)$ & $T:={\tt
GA}(\widetilde{G}(n))$ & $\hi(\langle S \rangle)$ &
$\hi(\langle T\rangle)$ \\
\hline \hline
2 & 6 & 5 & 0.12 & 0.05 & 0.51 & 0.68 \\
3 & 15 & 14 & 0.29  & 0.19 & 2.10 & 2.28 \\
4 & 28 & 27 & 0.86 & 1.12 & 11.28 & 12.77\\ \hline
\end{tabular}
\end{center}
\end{table}

\vspace{20pt}

\subsubsection*{Acknowledgment}
The authors were supported by Science Foundation Ireland grant
11/RFP.1/\abk~MTH3212.

\bibliographystyle{amsplain}

\end{document}